\documentclass[11pt,a4paper]{amsart}
\usepackage[english]{babel}
\usepackage[T1]{fontenc}
\usepackage{mathpazo,amssymb,bm}
\usepackage{upref}
\usepackage{enumerate}
\usepackage{graphicx,xcolor}
\usepackage[T1]{fontenc} 
\usepackage{cancel}
\usepackage{dsfont}       
\usepackage[colorlinks=true, pdftitle={Helmholtz solutions}, pdfauthor={Daniel Hauer and David Lee},hyperindex,breaklinks]{hyperref}
\usepackage{subfiles}

\title[Characterization of solutions]{Characterization of solutions of a generalized Helmholtz problem}

\author{Daniel Hauer} \address[Daniel Hauer]{School of Mathematics and
  Statistics, The University of Sydney, NSW 2006, Australia}
\email{\href{mailto:daniel.hauer@sydney.edu.au}{\nolinkurl{daniel.hauer@sydney.edu.au}}}

\author{David Lee} \address[David Lee]{Laboratoire Jacques-Louis Lions, Sorbonne Universit\'e, 75005 Paris, France}
\email{\href{mailto:david.lee@upmc.fr}{\nolinkurl{david.lee@upmc.fr}}}

\thanks{The first author's research was supported by an Australian Research
  Council grant DP200101065 and the second author's research was supported by the European Research
Council (ERC) under the European Union’s Horizon 2020 research and innovation programme (Grant Agreement n◦ 864066). The first author would also like to thank Ren\'e Schilling for useful discussions. Especially, in regards to Section \ref{SEC:BANACH}. The second author would also like to thank Lo\"ic Bethencourt, Antoine Gloria and Yoan Tardy for useful discussions. 
}



\numberwithin{equation}{section}

\newtheorem{theorem}{Theorem}[section]

\newtheorem{lemma}[theorem]{Lemma}
\newtheorem{corollary}[theorem]{Corollary}

\theoremstyle{definition}
\newtheorem{definition}[theorem]{Definition}

\newtheorem{remark}[theorem]{Remark}

\newcommand\R{{\mathbb{R}}}

\newcommand\C{{\mathbb{C}}}

\newcommand\dx{\mathrm{d}x }

\newcommand\dt{\mathrm{d}t }
\newcommand\td{\mathrm{d} }

\DeclareMathOperator{\supp}{supp}
\DeclareMathOperator{\loc}{loc}

\newcommand\abs[1]{\lvert#1\rvert}

\newcommand\norm[1]{\lVert#1\rVert}

\definecolor{darkred}{rgb}{0.7,0.1,0.1}

\begin{document}
\date{\today}
\maketitle

\tableofcontents

\begin{abstract}
  In this article, we classify all distributional solutions of $f(-\Delta)u=f(1)u$ where $f$ is a non-constant Bernstein function. Specifically, we show that the Fourier transform of $u$ is a single-layer distribution on the unit sphere. Examples of such operators include $(-\Delta)^\sigma$ (for $\sigma \in (0,1]$), $\log(1-\Delta)$ and $(-\Delta)^\frac{1}{2}\text{tanh}((-\Delta)^\frac{1}{2})$.  
\end{abstract}

%
%

\section{Introduction}

In this article, we consider the following variant of the \emph{Helmholtz equation}:
\begin{equation}\label{eqn:helmholtz_general}
    f(-\Delta)u=f(1)u,\quad \text{on $\mathbb{R}^d$},
\end{equation}
where $f$ is a non-constant Bernstein function and $f(-\Delta)$ is a pseudo-differential operator induced from $f$, see Definition \ref{Def:Bernstein} and Definition \ref{DEF:PSEUDO_BERNSTEIN}. Examples of such operators include $(-\Delta)^\sigma$ (for $\sigma \in (0,1]$), $\log(1-\Delta)$ and $(-\Delta)^\frac{1}{2}\text{tanh}((-\Delta)^\frac{1}{2})$. See also \cite[Section 16.8]{MR2978140} for more examples. 

The aim of this article is to characterize the solutions of \eqref{eqn:helmholtz_general}. In particular, we show that the Fourier transform of solutions of \eqref{eqn:helmholtz_general} is the set of single-layer distributions on the unit sphere. 

This result can be viewed as an extension of the results by Fall and Weth \cite{MR3511811}, Guan, Murugan and Wei \cite{MR4546886} and Cheng, Li and Yang \cite{cheng2022localization,cheng2022equivalence}.

We will now introduce some notation.

\subsection{Notation and function spaces}
We denote $\mathbb{N}$ to be the set of natural numbers and $\mathbb{N}_0:=\mathbb{N}\cup\{0\}$ and $\mathbb{N}_0^d:=\prod_{i=1}^d\mathbb{N}_0$. Partial derivatives of a differentiable function $\varphi:\mathbb{R}^d\rightarrow \mathbb{C}$ are denoted by 
\begin{displaymath}
\partial^\alpha\varphi:=\partial_{x_1}^{\alpha_1}...\partial_{x_d}^{\alpha_d}\varphi, \quad \text{for all $\alpha=(\alpha_1,..,\alpha_d)\in \mathbb{N}_0^d $}. 
\end{displaymath}
We also denote for $\alpha=(\alpha_1,..,\alpha_d) \in \mathbb{N}_{0}^d$, $|\alpha|=\sum_{i=1}^d\alpha_i$. 

We denote $\mathcal{S}(\mathbb{R}^d)$ and $\mathcal{S}'(\mathbb{R}^d)$ to be the set of Schwartz functions and the set of tempered distributions respectively. We also denote $\mathcal{E}(\mathbb{R}^d)$ to be the set of smooth functions on $\mathbb{R}^d$ and the corresponding dual space $\mathcal{E}'(\mathbb{R}^d),$ to be the set of distributions with compact support. Moreover, we denote the Fourier transform of $\varphi$ as $\hat{\varphi}$ or $\mathcal{F}(\varphi)$ as well as the Fourier inverse as $\mathcal{F}^{-1}(\varphi)$. 

We take the convention, if $\varphi\in \mathcal{S}(\R^d),$
\begin{displaymath}
    \mathcal{F}(\varphi)(\xi):=\int_{\mathbb{R}^d}f(x)\,e^{-ix\cdot \xi}\,dx,\quad \mathcal{F}^{-1}(\varphi)(x):=(2\pi)^{-d}\int_{\mathbb{R}^d}f(x)\,e^{ix\cdot \xi}\,d\xi. 
\end{displaymath}
Similar notation will also be used on $\mathcal{S}'(\R^d).$

We also denote 
\begin{displaymath}
    \mathcal{Z}(\mathbb{R}^d):=\{\varphi \in \mathcal{S}(\R^d):\partial^\alpha\hat{\varphi}(0)=0, \text{for all }\alpha\in \mathbb{N}_0^d\},
\end{displaymath}
as well as the corresponding topological dual $\mathcal{Z}'(\mathbb{R}^d)$, the set of all continuous functionals on $\mathcal{Z}(\R^d)$. Note that for every $u\in \mathcal{S}'(\R^d)$ then the restriction $u|_{\mathcal{Z}(\R^d)}\in \mathcal{Z}'(\R^d).$ Conversely, for $u\in \mathcal{Z}'(\R^d)$, an extension to $\mathcal{S}'(\R^d)$ exists. The extension is unique modulo the set of polynomials. 

Specifically, one can identify $\mathcal{Z}'(\R^d)$ as the quotient space 
\begin{displaymath}
    \mathcal{S}'(\R^d)/\mathbb{C}([x_1,..,x_d]), \quad \text{where $\mathbb{C}([x_1,..,x_d])$ is the linear space of polynomials.}
\end{displaymath}

Moreover, we define $\mathcal{D}(\mathbb{S}^{d-1})$ to be the set of smooth functions on the unit sphere $\mathbb{S}^{d-1}$ and $\mathcal{D}'(\mathbb{S}^{d-1})$ to be the set of distributions on $\mathbb{S}^{d-1}$, see Definition \ref{DEF:spherical_Dist}. 

Moreover, we denote $e_{x}(y):=e^{-ix\cdot y},$ for $y\in \mathbb{R}^{d}$ and $x\in \R^d$.

\subsection{Main Result}
Before we define our notion of solution for \eqref{eqn:helmholtz_general} we remind the reader of the definition of a Bernstein function. 

\begin{definition}\label{Def:Bernstein}
	A function $f:[0,\infty)\rightarrow \mathbb{R}$ is a Bernstein function if $f$ is smooth on $(0,\infty)$, and 
	\begin{displaymath}
		(-1)^{n-1}f^{(n)}(\lambda)\geq 0, \quad \text{for all $n\in \mathbb{N}$ and $\lambda >0$}. 
	\end{displaymath} 
\end{definition}

An equivalent definition due to the \emph{L\'evy triple} \cite[Theorem 3.2]{MR2978140} will be more useful for us. 

\begin{theorem}\ \\
	A function $f:[0,\infty)\rightarrow \mathbb{R}$ is a Bernstein function if, and only if, it admits the representation
	\begin{displaymath}
		f(\lambda)=a+b\lambda + \int_{0}^{\infty} (1-e^{-\lambda t})\, \mu(dt),\quad \lambda>0,
	\end{displaymath}
	where $a,b\geq 0$ and $\mu$ is a measure on $(0,\infty)$ such that 
	\begin{displaymath}
		\int_{(0,\infty)}(1\wedge t) \,\mu(dt)<\infty. 
	\end{displaymath}
	In particular, $f$ uniquely determines $(a,b,\mu)$ and vice-versa. 
\end{theorem}

We now define the operator $f(-\Delta)$. 
\begin{definition}\label{DEF:PSEUDO_BERNSTEIN}
Let $f$ be a Bernstein function with triple $(a,b,\mu)$.
We define $f(-\Delta):\mathcal{Z}(\R^d)\rightarrow \mathcal{Z}(\R^d)$ by 
\begin{displaymath}
    \begin{split}
        f(-\Delta):&\mathcal{Z}(\R^d)\rightarrow \mathcal{Z}(\R^d),\\
        &\varphi\mapsto \mathcal{F}^{-1}(f(|.|^2)\hat{\varphi}).
    \end{split}
\end{displaymath}
Moreover, we define $f(-\Delta):\mathcal{Z}'(\R^d)\rightarrow \mathcal{Z}'(\R^d)$ by 
\begin{displaymath}
    \left \langle f(-\Delta)u,\varphi \right \rangle_{\mathcal{Z}'(\R^d),\mathcal{Z}(\R^d)}:=    \left \langle u,f(-\Delta)\varphi \right \rangle_{\mathcal{Z}'(\R^d),\mathcal{Z}(\R^d)},
\end{displaymath}
for all $(u,\varphi)\in \mathcal{Z}'(\R^d)\times \mathcal{Z}(\R^d).$
\end{definition}

The reader might wonder why we need the space $\mathcal{Z}'(\R^d)$ as opposed to tempered distributions $\mathcal{S}'(\R^d)$. The reason being is that $f(-\Delta)$, in general, is not well-defined on $\mathcal{S}'(\R^d)$. Specifically, it does not belong to classical class of pseudo-differential operators (for instance see \cite{MR2304165}). We direct the reader to Appendix \ref{SUBSEC:NEED_Z} for more details about the considerations needed when defining a distributional definition of $f(-\Delta)$. We also direct the reader to Appendix \ref{SEC:APPENDIX} to show that the above definition is indeed well-defined. 

Let us now be more precise about our notion of solution for \eqref{eqn:helmholtz_general}. 

\begin{definition}\label{Def:Helmholtz}
    We say that $u$ is a solution to \eqref{eqn:helmholtz_general} if $u\in \mathcal{Z}'(\R^d)$ and
    \begin{equation}\label{eqn:def_soln}
        \left \langle u, f(-\Delta)\varphi\right \rangle_{\mathcal{Z}'(\R^d),\mathcal{Z}(\R^d)}=\left \langle u, f(1)\varphi\right \rangle_{\mathcal{Z}'(\R^d),\mathcal{Z}(\R^d)}, \text{for all $\varphi \in \mathcal{Z}(\R^d)$}. 
    \end{equation}
    and $0\notin \supp{\hat{u}}$.
\end{definition}

   The condition $0\notin \text{supp}(\hat{u})$ prevents us from having polynomials as potential solutions. Moreover, it allows us to \emph{uniquely} identify $\hat{u}\in \mathcal{S}'(\R^d)$, in the sense that we can choose a unique element in the equivalence class of $[\hat{u}]\in \mathcal{S}'(\R^d)/\mathbb{C}([x_1,..,x_d])$ such that $0\notin \supp(\hat{u})$. 

\begin{remark}\label{rem:equi}
    Note that \eqref{eqn:def_soln} is equivalent to
    \begin{displaymath}
        \left \langle \hat{u}, (f(|.|^2)-f(1))\hat{\varphi}\right \rangle_{\mathcal{S}'(\R^d),\mathcal{S}(\R^d)}=0, \text{for all $\varphi \in \mathcal{Z}(\R^d)$}. 
    \end{displaymath}
    by the Parseval Relation.
\end{remark}

Before we state our main result we define the following vector space:
    \begin{displaymath}
    V_{f(-\Delta)}:=\{u\in \mathcal{Z}'(\R^d):\text{$u$ solves $f(-\Delta)u=f(1)u$ in the sense of Definition \ref{Def:Helmholtz}}\}. 
    \end{displaymath}

\begin{theorem}[Main Result]\label{main_thm}
    Let $f$ be a non-constant Bernstein function.
    Then $u\in V_{f(-\Delta)}$ if and only if 
    \begin{displaymath}
        u(x)=\left \langle T, e_{x}|_{\mathbb{S}^{d-1}}\right \rangle_{\mathcal{D}'(\mathbb{S}^{d-1}),\mathcal{D}(\mathbb{S}^{d-1})}, \quad \text{for all $x\in \mathbb{R}^d$},
    \end{displaymath}
    for some $T\in \mathcal{D}'(\mathbb{S}^{d-1})$. 
\end{theorem}

The above theorem, in some sense, generalizes the existing result of the classical Helmholtz problem, $f(-\Delta)=(-\Delta)$.
\begin{theorem}\label{thm:laplace_case}
$u\in V_{(-\Delta)}$ if and only if 
    \begin{equation}\label{eqn:intro_exact}
        u(x)= \left \langle T, e_{x}|_{\mathbb{S}^{d-1}}\right \rangle_{\mathcal{D}'(\mathbb{S}^{d-1}),\mathcal{D}(\mathbb{S}^{d-1})}, \quad x\in \mathbb{R}^d,
    \end{equation}
for some $T\in \mathcal{D}'(\mathbb{S}^{d-1})$.
\end{theorem}

From the above result, we have the following equivalent formulation of Theorem \ref{main_thm}. 
\begin{theorem}[Equivalent formulation]\label{main_thm_2}
    Let $f$ be a non-constant Bernstein function. Then 
    \begin{displaymath}
    V_{f(-\Delta)}=V_{(-\Delta)}.    
    \end{displaymath}
\end{theorem}

The characterization given in Theorem \ref{thm:laplace_case} appears explicitly in \cite{MR1730501} however it was almost surely known before within the context of the \emph{Fourier Restriction problem}. The literature on this subject is incredibly vast so we direct the reader to the survey of Tao \cite{MR2087245}.

However, leveraging existing results within the Fourier Restriction problem for the sphere and Theorem \ref{main_thm_2} we can give a more precise characterization. 

For this, we define the space $B^{\ast}(\R^d)$, which is defined as the space of all $u\in L^2_{\loc}(\R^d)\cap \mathcal{S}'(\R^d)$ satisfying
\begin{displaymath}
    \norm{u}^2_{B^{\ast}(\R^d)}=\sup_{R>1}\frac{1}{R}\int_{\{\abs{x}<R\}}
    \abs{u}^2\,\dx\,<\infty. 
\end{displaymath}

The following corollary is obtained from Theorem \ref{main_thm_2} and \cite[Theorem 2.2]{MR466902}/ \cite[Lemma 3.2]{MR1025883}. 

\begin{corollary}
    Suppose that $u\in V_{f(-\Delta)}$ where $f$ is a non-constant Bernstein function. Denote $\sigma$ to be the surface measure of $\mathbb{S}^{d-1}$. Then $u \in B^{\ast}(\R^d)$ if and only if 
        \begin{displaymath}
     u(x)=\int_{\mathbb{S}^{d-1}}\Phi(\xi)\,e^{-i x\cdot \xi}\,\td\sigma(\xi)\quad\text{for all every $x\in \R^d$,}
   \end{displaymath}
for some $\Phi\in L^2(\mathbb{S}^{d-1})$.
\end{corollary}

In this article we present 2 proofs of Theorem \ref{main_thm}. The case where $\mu=0$ and $b$ is non-zero is well known, see Section \ref{sec:classical_division}. So for the rest of this article we will work with Bernstein functions of the form $(a,b,\mu)$ where $\mu$ is non-zero.

\section{Outline of the article}

We begin with a comparison of our results and existing literature in Section \ref{SEC:LITERATURE}. 

In Section \ref{SEC:some_complex_analysis} we investigate the existence of smooth extensions of the functions
\begin{displaymath}
        \omega(z):=\frac{f(z)-f(1)}{z-1} ,\quad \text{and} \quad \frac{1}{\omega (z)}:=\frac{z-1}{f(z)-f(1)} ,
    \end{displaymath}
    via complex analysis which will be useful for later proofs. In Section \ref{SEC:DIST_SPHERE} we introduce the space of distributions on the sphere and a particular case of the classical structure theorem. 

Section \ref{sec:proof} contains the first and second proofs of the main result, Theorem \ref{main_thm}.

The proof outlined in Section \ref{sec:first_proof} focuses on establishing that $u$ solves
\begin{displaymath}
    f(-\Delta)u=f(1)u, \quad \text{if and only if }\quad (-\Delta)u=u. 
\end{displaymath}
From Theorem \ref{main_thm_2} we can obtain Theorem \ref{main_thm}. 

Section \ref{SEC:second_proof} focuses on a more direct approach and utilizes the classical structure theorems, see Theorem \ref{thm:classical_strc}.

Section \ref{SEC:BANACH} focuses on a Banach space version of the main result and some future questions.

In Appendix \ref{SUBSEC:NEED_Z} and \ref{SEC:APPENDIX}, we focus on the details on establishing that $f(-\Delta):\mathcal{Z}'(\R^d)\rightarrow \mathcal{Z}'(\R^d)$ is a well-defined and continuous mapping.

\section{Review of the literature} \label{SEC:LITERATURE}
We remind the reader of the classical connection of homogeneous problems and algebraic varieties. 

\subsection{The classical division problem}\label{sec:classical_division}
Take $P$ to be polynomial on $\R^d$ and consider the following homogeneous problem on $\mathbb{R}^d$:
\begin{displaymath}
    P(i\nabla)u=0, \quad \text{in $\mathcal{S}'(\R^d)$,}
\end{displaymath}
then, via the Fourier transform, it is equivalent to  solving the \emph{classical division problem}:
\begin{displaymath}
    P\hat{u}=0, \quad \text{in $\mathcal{S}'(\R^d)$.}
\end{displaymath}

In fact, if $P$ has real coefficients and real simple zeros we have that $\hat{u}$ is a single layer distribution on the real algebraic variety $\{\xi\in \R^d:P(\xi)=0\}$, see Remark \ref{remark:layer_dist}.

For the case where  $P(\xi)=|\xi|^2-1$, for $\xi \in \mathbb{R}^d$, we have the \emph{classical Helmholtz problem} and the following equivalence: 
\begin{displaymath}
    \Delta u+u=0\quad \text{in $\mathcal{S}'(\R^d)$}, \quad \text{if and only if} \quad (|.|^2-1)\hat{u}=0, \quad \text{in $\mathcal{S}'(\R^d)$. }
\end{displaymath}

The latter equation has a well-known solution. Specifically, $\hat{u}$ necessarily is a \emph{single-layer distribution} on $\mathbb{S}^{d-1}:=\{\xi\in \mathbb{R}^d:|\xi|=1\},$ see \cite[\emph{exemple} on page 128]{MR0209834} and Definition \ref{DEF:layer_dist}. 

It is from this classification that one can obtain an explicit characterization for temperate solutions of the classical Helmholtz problem, in particular Theorem \ref{thm:laplace_case}.

\begin{remark}\label{remark:layer_dist}
    For the interested reader who is interested in the theory of distributions on algebraic varieties (or more generally on hyper surfaces) see the classical book of Laurent Schwartz \cite[page 127-128]{MR0209834}, the following articles by Agmon and H\"ormander \cite{MR466902} or Wagner \cite{MR2681439} for a more recent treatment. 
\end{remark}

\subsection{A review of non-local variants of the classical division problem}

Recently, there has been a surge of interest in the classification of non-local variants of the Helmholtz solutions.

Fall and Weth \cite{MR3511811} established a Liouville theorem for a family of non-local problems associated with a subclass of L\'evy operators. 

Specifically, they showed that if $\mathcal{L}_{\nu}$ is a L\'evy operator, see \cite{MR3511811} for more precise conditions on $\mathcal{L}_{\nu}$, then we have that the distributional support of the solution of 
\begin{equation}\label{eqn:fall_weth}
    \mathcal{L}_{\mu}u+P(i\nabla)u=0, \text{on $\R^d$,}
\end{equation}
where $P$ is a complex polynomial, is contained in the set $\{\xi\in \R^d:\eta(\xi)+P(\xi)=0\},$ where $\eta$ is the corresponding Fourier symbol of $\mathcal{L}_{\nu}.$

Let us now explain some subleties of their result by considering a specific application of their result. 

In \cite[Theorem 4.2]{MR3511811}, they showed that the one dimensional case allows them to identify that
\begin{equation}\label{eqn:frac_intro_1}
    V_{(-\Delta)}\cap L^\infty(\R)=V_{(-\Delta)^\sigma}\cap L^\infty(\R)
\end{equation}
for $\sigma\in(0,1)$.

The result is established by showing that if $u$ solves \eqref{eqn:frac_intro_1} then $\supp(\hat{u})\subset \{-1,1\}$. Hence, by the classical structure theorems \cite[Chapter 6]{MR1276724}, $\hat{u}$ must be of the form 
\begin{displaymath}
    \hat{u}=\sum_{\alpha\leq N}a_{\alpha}\partial^\alpha \delta_{1}+\sum_{\beta\leq N}b_{\alpha}\partial^\alpha \delta_{-1}, 
\end{displaymath}
for constants $a_{\alpha},b_{\alpha}\in \R$ and $N\in \mathbb{N}_0$, where $\delta_{z}$ denotes the dirac measure at $z\in \R$.
Hence, if $u$ is bounded then necessarily 
\begin{equation}\label{eqn:frac_intro_2}
    \hat{u}=a\delta_{1}+b\delta_{-1}, 
\end{equation}
for some $a,b\in \R$.
 This is allows us to directly obtain \eqref{eqn:intro_exact} (for $d=1$).

It is from this characterization that they were able to establish the equivalence of the Helmholtz problem and the fractional Helmholtz problem for bounded solutions on $\mathbb{R}$. 

Generalizations of the above result have been done by Guan, Murugan and Wei \cite{MR4546886} and Cheng, Li and Yang \cite{cheng2022localization,cheng2022equivalence}.

The following identification was done by Guan, Murugan and Wei \cite{MR4546886}
\begin{align*}
&V_{f(-\Delta)}\cap \{u\in L^2(\mathbb{R}^d)\cap L^\infty(\mathbb{R}^d):\lim_{|x|\rightarrow \infty}u(x)=0\}\\
&=V_{(-\Delta)}\cap \{u\in L^2(\mathbb{R}^d)\cap L^\infty(\mathbb{R}^d):\lim_{|x|\rightarrow \infty}u(x)=0\},    
\end{align*}
for a subset of \emph{complete Bernstein functions}, using the extension technique of Kwa\'{s}nicki and Mucha \cite{MR3859452}. 

More recently, Cheng, Li and Yang, in \cite[$f(-\Delta)=(-\Delta)^\sigma$, for $\sigma\in (0,1)$]{cheng2022equivalence} and \cite{cheng2022localization}, established that 
\begin{displaymath}
    V_{f(-\Delta)}\cap L^\infty(\mathbb{R}^d)=V_{(-\Delta)}\cap L^\infty(\mathbb{R}^d),
\end{displaymath}
for a class of $f$ more general than the class of Bernstein functions by more traditional Fourier methods. 

We should stress that the difference with our result and the one of Fall and Weth \cite{MR3511811} is that we are able to identify the Fourier transform of our generalized Helmholtz problem as a \emph{single-layer distribution} on $\mathbb{S}^{d-1}$. The Liouville theorem in \cite{MR3511811} gives a characterization of the support of the Fourier transform of the solution of \eqref{eqn:fall_weth} but doesn't necessarily imply that we have a single-layer distribution on a hypersurface.

Both of our proofs takes inspiration from the celebrated article of Strichartz \cite{MR1025883}. However, we note that the first proof presented here is similar to the proofs presented in \cite{cheng2022localization} and \cite{cheng2022equivalence}. However, the difference with our first proof and the proofs presented in \cite{cheng2022localization} and \cite{cheng2022equivalence} is that, keeping in spirit with Fall and Weth \cite{MR3511811}, we are very insistent on the use of the space $\mathcal{Z}'(\R^d)$ to allow us to consider unbounded solutions. For this reason we believe that it is valuable for us to present this first proof. 

\begin{remark}
Since solutions of the classical helmholtz solutions are typically periodic the reader might also wonder whether if we could consider the setting of periodic solutions. However, this is trivial since the spectrum of $(-\Delta)$ becomes discrete and the result holds almost immediately. This is also observed in \cite[proof of Theorem 13.50]{MR2978140}.
\end{remark}

\begin{remark}
Recent approaches to proving non-local analogues of Liouville theorems, see for instance \cite{MR4472712,MR4149690,grzywny2023liouvilles,MR3511811,berger2022liouville}, also define non-local operators in the same spirit (but not exactly) as in Definition \ref{DEF:PSEUDO_BERNSTEIN}. 

We believe that it would be useful if Theorem \ref{thm:appendix} could be be extended to L\'evy type operators but we are unsure if this has already been established outside of the case where $f(-\Delta)=(-\Delta)^\sigma$, for instance see \cite[Section 5.1.2]{MR3024598}.  
\end{remark}


\section{Some complex analysis}\label{SEC:some_complex_analysis}
In this section we wish to determine whether the symbols associated with our pseudo-differential operators are infinitely differentiable. In order to do this it will be more convenient for us to prove that their complex extensions are holomorphic. 

For now we take $f$ to be a Bernstein function with the triple $(a,b,\mu)$ where $\mu$ is non-zero. 

Note that $f$ has a holomorphic extension on $\mathbb{H}:=\{z\in \mathbb{C}:\text{Re}(z)>0\}$, see \cite[Proposition 3.6]{MR2978140} and $f'(1)\ne 0$. 

We now define $\omega:\mathbb{H}\setminus\{1\}\rightarrow \mathbb{C}$ as
    \begin{displaymath}
        \omega(z):=\frac{f(z)-f(1)}{z-1}, \quad \text{for all $z\in\mathbb{H}\setminus\{1\}$}. 
    \end{displaymath}
    We now show that $\omega$ has a holomorphic extension to $\mathbb{H}$.

\begin{lemma}\label{Lemma:function}
$\omega$ is holomorphic on $\mathbb{H}\setminus\{1\}$ and has a holomorphic extension on $\mathbb{H}$. 
\end{lemma}

\begin{proof}
    $\omega$ is clearly holomorphic on $\mathbb{H}\setminus\{1\}$ by the complex quotient rule. We now only have to prove that $\omega$ has a complex extension on $\mathbb{H}$. 
    By \cite[Chapter VI, Theorem 1.2]{MR1344449}, as well as using the fact that $f$ is holomorphic, we have that $\omega$ has a removable singularity at $z=1$ and hence a holomorphic extension on $\mathbb{H}$.
\end{proof}

We now take $\omega$ to be the holomorphic extension. Specifically, we take $\omega(1):=f'(1)\ne 0.$
It will also be useful for us to consider whether $\frac{1}{\omega}$ is holomorphic. 
\begin{lemma}
    We have that there exists an $\epsilon>0$ such that $\frac{1}{\omega}$ is holomorphic on $\{z\in \mathbb{H}:|z-1|<\epsilon\}$. 
\end{lemma}

\begin{proof}
    By the complex quotient rule we have that $1/\omega$ is holomorphic on the set $\{z\in \mathbb{H}:\omega(z)\ne 0\}$ which is an open set. Note that $\omega(1)=f'(1)\ne 0$. Hence, since $\{z\in \mathbb{H}:\omega(z)\ne 0\}$ is an open set, we must have that there exists an $\epsilon >0$ such that $\{z\in \mathbb{H}:|z-1|<\epsilon\}\subset \{z\in \mathbb{H}:\omega(z)\ne 0\}$. 
\end{proof}

\begin{lemma}\label{Lemma:smooth}
    Both the following restrictions of $\omega$ and $\frac{1}{\omega}$, $\omega:(0,\infty)\rightarrow \mathbb{R}$ and $\frac{1}{\omega}:(1-\epsilon,1+\epsilon)\rightarrow \mathbb{R}$ are infinitely differentiable. 
\end{lemma}

\begin{proof}
    Note since $\omega:\mathbb{H}\rightarrow \mathbb{C}$ and $\frac{1}{\omega}:\{z\in \mathbb{H}:|z-1|<\epsilon\}\rightarrow \mathbb{C}$ are both holomorphic they are necessarily infinitely differentiable as complex functions. However, clearly this implies that both $\omega:(0,\infty)\rightarrow \mathbb{R}$ and $\frac{1}{\omega}:(1-\epsilon,1+\epsilon)\rightarrow \mathbb{R}$ are infinitely differentiable as real functions. 
\end{proof}


\section{Distributions on the sphere and the classical structure theorem}
\label{SEC:DIST_SPHERE}

We introduce the space of distributions on $\mathbb{S}^{d-1}.$ For us it will more useful to consider the extrinsic definition as opposed to the intrinsic one, see \cite[Section 6.7]{MR1276724} for more details. 

Before we introduce the definitions we choose $\epsilon\in (0,1)$ to  be defined as in Lemma \ref{Lemma:smooth} and $\rho$ to be a smooth radial bump function such that 
    \begin{equation}\label{eqn:rho}
        \rho=\begin{cases}
            1, \quad \text{for $1-\frac{\epsilon}{4}\leq |x|\leq 1+\frac{\epsilon}{4}$},\\
            0, \quad \text{for $|x|\leq 1-\frac{\epsilon}{2}, |x|\geq 1+\frac{\epsilon}{2}$}.\\
        \end{cases}
    \end{equation}
We now define the space of distributions on $\mathbb{S}^{d-1}$. 
\begin{definition}\label{DEF:spherical_Dist}
    Let $\varphi$ be a function defined on $\mathbb{S}^{d-1}$. Define 
    \begin{displaymath}
        E\varphi(x):=\begin{cases}
            &\rho(|x|)\varphi(x/|x|), \quad \text{for $|x|\in (1-\frac{\epsilon}{2}, 1+\frac{\epsilon}{2})$},\\
            &0, \qquad\qquad\qquad\quad  \text{otherwise}.
        \end{cases}
    \end{displaymath}
    Then we say $\varphi \in\mathcal{D}(\mathbb{S}^{d-1})$ if and only if $E\varphi \in \mathcal{D}(\mathbb{R}^d).$ Moreover, we say that 
    \begin{displaymath}
        \lim_{n\rightarrow \infty}\varphi_n=\varphi, \quad \text{in $\mathcal{D}(\mathbb{S}^{d-1})$,}
    \end{displaymath}
    if and only if 
    \begin{displaymath}
        \lim_{n\rightarrow \infty}E\varphi_n=E\varphi, \quad \text{in $\mathcal{D}(\mathbb{R}^d).$}
    \end{displaymath}
    We also define $\mathcal{D}'(\mathbb{S}^{d-1})$ as the space of continuous linear functionals on $\mathcal{D}(\mathbb{S}^{d-1})$. 
\end{definition}

A natural question is whether $T\in \mathcal{D}'(\mathbb{S}^{d-1})$ can be considered a distribution on $\mathcal{D}'(\mathbb{R}^{d})$ and vice versa. This leads us to the notion of \emph{single or multi-layer distributions} on $\mathbb{S}^{d-1}.$

\begin{definition}\label{DEF:layer_dist}
    Let $T\in \mathcal{D}'(\mathbb{S}^{d-1}).$ We call $L^{(0)}_{\mathbb{S}^{d-1}}(T)\in \mathcal{D}'(\mathbb{R}^d)$, the \emph{single-layer distribution} on $\mathbb{S}^{d-1}$ with density of $T$ such that 
    \begin{displaymath}
        \left \langle L^{(0)}_{\mathbb{S}^{d-1}}(T),\varphi\right \rangle_{\mathcal{D}'(\mathbb{R}^d),\mathcal{D}(\mathbb{R}^d)}
        =
        \left \langle T,\varphi|_{\mathbb{S}^{d-1}}\right \rangle_{\mathcal{D}'(\mathbb{S}^{d-1}),\mathcal{D}(\mathbb{S}^{d-1})}, \quad \text{for all $\varphi \in \mathcal{D}(\R^d)$}.
    \end{displaymath}
    More generally, for $k\in \mathbb{N},$  we define the \emph{multi-layer distribution} $L_{\mathbb{S}^{d-1}}^{(k)}(T)\in \mathcal{D}'(\mathbb{R}^d)$ by 
    \begin{displaymath}
        \left \langle L_{\mathbb{S}^{d-1}}^{(k)}(T),\varphi\right \rangle_{\mathcal{D}'(\mathbb{R}^d),\mathcal{D}(\mathbb{R}^d)}
        =
        (-1)^k\left \langle T,D^k_{\nu}\varphi|_{\mathbb{S}^{d-1}}\right \rangle_{\mathcal{D}'(\mathbb{S}^{d-1}),\mathcal{D}(\mathbb{S}^{d-1})}, 
    \end{displaymath}
    for all $\varphi \in \mathcal{D}(\R^d)$, where $D^k_{\nu}$ is the $k$-th (outward) normal derivative of $\varphi$. Specifically,
    \begin{displaymath}
        D_{\nu}\varphi|_{\mathbb{S}^{d-1}}=\frac{1}{|x|}x\cdot \nabla \varphi|_{\mathbb{S}^{d-1}}=x\cdot \nabla \varphi|_{\mathbb{S}^{d-1}}.
    \end{displaymath}
\end{definition}

It is straightforward to check that single or multi-layer distributions on $\mathbb{S}^{d-1}$ are supported on $\mathbb{S}^{d-1}$. It should be noted that, whilst single-layer distributions can be identified with a distribution on $\mathbb{S}^{d-1}$ (and vice-versa), multi-layer distributions cannot be identified on $\mathcal{D}'(\mathbb{S}^{d-1})$ since the value 
 \begin{displaymath}
    D^k_{\nu}\varphi|_{\mathbb{S}^{d-1}}, \quad k\in \mathbb{N},
\end{displaymath}
depends very much on how one defines $\varphi$ in a local neighbourhood of $\mathbb{S}^{d-1}$ as opposed to the single-layer case since
 \begin{displaymath}
    \varphi|_{\mathbb{S}^{d-1}}, 
\end{displaymath}
is very much independent of how $\varphi$ is realized on $\mathbb{R}^d$, see \cite[page 100]{MR1276724} or \cite{MR2681439}. 

The classical structure theorem shows us that distributions supported on $\mathbb{S}^{d-1}$ are a linear combination of single and multi-layer distributions, see \cite[page 102]{MR0209834} or \cite[Theorem 6.7.1]{MR1276724}. 

\begin{theorem}\label{thm:classical_strc}
    Every distribution $T\in \mathcal{D}'(\mathbb{R}^d)$ supported on $\mathbb{S}^{d-1}$ can be uniquely written as 
    \begin{displaymath}
        T=\sum_{k=0}^N L_{\mathbb{S}^{d-1}}^{(k)}(T_k),
    \end{displaymath}
    for some $N\in \mathbb{N}_0$, $T_k\in \mathcal{D}'(\mathbb{S}^{d-1})$ for $k=0,..,N$. 
\end{theorem}


\section{Proofs of the main result}
\label{sec:proof}

Throughout this section we define $\omega:(0,\infty)\rightarrow \mathbb{R}$ and $\frac{1}{\omega}:(1-\epsilon,1+\epsilon)\rightarrow \mathbb{R}$ as in Lemma \ref{Lemma:smooth}. Moreover, we define $\rho$ as in \eqref{eqn:rho}. 

The following fact will be useful for both of our proofs. We now show that if $u$ solves $f(-\Delta)u=f(1)u$, in the sense of Definition \ref{Def:Helmholtz}, then the Fourier transform of $u$ has compact support contained in $\mathbb{S}^{d-1}$. 
\begin{lemma}\label{lemma:support}
    Suppose that $u$ solves $f(-\Delta)u=f(1)u$. Then $\supp(\hat{u})\subset \mathbb{S}^{d-1}.$
 \end{lemma}
\begin{proof}
Our aim will to be show that 
\begin{displaymath}
    \langle \hat{u},\hat{\psi}\rangle_{\mathcal{D}'(\R^d),\mathcal{D}(\R^d)}=0,
\end{displaymath}
for $\hat{\psi}\in \mathcal{D}(\mathbb{R}^d\setminus (\mathbb{S}^{d-1}\cup\{0\}))$. It is clear that necessarily $\psi\in \mathcal{Z}.$
By Definition \ref{Def:Helmholtz}, we have that 
\begin{equation}\label{eqn:transform}
\left \langle \hat{u}, (f(|.|^2)-f(1))\hat{\psi})\right \rangle_{\mathcal{Z}'(\R^d),\mathcal{Z}(\R^d)}=\left \langle \hat{u}, (f(|.|^2)-f(1))\hat{\psi})\right \rangle_{\mathcal{D}'(\R^d),\mathcal{D}(\R^d)}=0.
\end{equation}

Then we have that 
\begin{equation}\label{eqn:test}
    \left \langle \hat{u}, \hat{\psi}\right \rangle_{\mathcal{D}'(\R^d),\mathcal{D}(\R^d)}=\left \langle \hat{u}, (f(|.|^2)-f(1)) \tilde{\psi}\right \rangle_{\mathcal{D}'(\R^d),\mathcal{D}(\R^d)}
\end{equation}
where 
\begin{displaymath}
    \tilde{\psi}(\xi)=\begin{cases}
        &\frac{1}{f(|\xi|^2)-f(1)}\hat{\psi}(\xi),\quad  \text{for all $\xi\in \supp(\psi)$}\\
        &0,\hspace{2.78cm} \text{otherwise}. 
    \end{cases}
\end{displaymath}  
It is clear that $\tilde{\psi}$ has compact support. Moreover, we see that $\frac{1}{f(|.|^2)-f(1)}$ is smooth and bounded above and below on $\supp(\psi)$ since $f$ is injective and $\supp(\psi)\subset \mathbb{R}^d\setminus \left (\mathbb{S}^{d-1}\cup\{0\}\right)$. Hence, we must have that $\mathcal{F}^{-1}(\tilde{\psi})\in \mathcal{Z}(\R^d)$ and so by \eqref{eqn:transform} and \eqref{eqn:test} we must have that 
$\hat{u}=0$ on $\mathbb{R}^d\setminus \left(\mathbb{S}^{d-1}\cup\{0\} \right)$ in the distributional sense. However, by Definition \ref{Def:Helmholtz}, we have that $\hat{u}=0$ on $\mathbb{R}^d\setminus \mathbb{S}^{d-1}$ hence $\supp(\hat{u})\subset \mathbb{S}^{d-1}.$
\end{proof}

\subsection{First proof}\label{sec:first_proof}

\begin{lemma}\label{Lemma:proof_1_1}
     Suppose that $u$ solves $f(-\Delta)u=f(1)u$ in the sense of Definition \ref{Def:Helmholtz}. Then $(-\Delta)u=u$ in $\mathcal{S}'(\R^d)$.
\end{lemma}

\begin{proof}
We identify $\hat{u}\in \mathcal{Z}'(\R^d)$ by its unique extension in $\mathcal{S}'(\R^d),$  note this is possible by the assumption $0\notin \supp(\hat{u}).$
By definition, we wish to show that 
    \begin{displaymath}
        \left\langle(-\Delta)u-u,\varphi\right\rangle_{\mathcal{S}'(\R^d),\mathcal{S}(\R^d)}=\langle \hat{u},(|.|^2-1)\hat{\varphi}\rangle_{\mathcal{S}'(\R^d),\mathcal{S}(\R^d)}=0, \quad \text{for all $\varphi \in \mathcal{S}(\R^d)$}. 
    \end{displaymath}

It is clear that $\frac{1}{\omega}(|.|^2)\rho$ has a natural extension on $\mathbb{R}^d$ by defining $\frac{1}{\omega}(|.|^2)\rho=0$ for $|x|\leq 1-\frac{\epsilon}{2}$ and $|x|\geq 1+\frac{\epsilon}{2}$. From this construction we can see that $\frac{1}{\omega}(|.|^2)\rho$ is a smooth function with compact support and $\mathcal{F}^{-1}(\frac{1}{\omega}(|.|^2)\rho)\in \mathcal{Z}(\R^d)$. 
    
With this in mind we have that 
    \begin{displaymath}
    \begin{split}
        \langle \hat{u},(|.|^2-1)\varphi\rangle_{\mathcal{S}'(\R^d),\mathcal{S}(\R^d)}&=\langle \hat{u},(|.|^2-1)\varphi \rho\rangle_{\mathcal{S}'(\R^d),\mathcal{S}(\R^d)}\\
        &+\langle \hat{u},(|.|^2-1)\varphi (1-\rho)\rangle_{\mathcal{S}'(\R^d),\mathcal{S}(\R^d)}.
    \end{split}
    \end{displaymath}
    Hence, 
    \begin{displaymath}
        \langle \hat{u},(|.|^2-1)\varphi\rangle_{\mathcal{S}'(\R^d),\mathcal{S}(\R^d)}=\langle \hat{u},(|.|^2-1)\varphi \rho\rangle_{\mathcal{S}'(\R^d),\mathcal{S}(\R^d)},
    \end{displaymath}
    since $\supp(\hat{u})\subset \mathbb{S}^{d-1}$, by Lemma \ref{lemma:support}. 
 Moreover, we have that 
    \begin{displaymath}
        (|.|^2-1)\varphi \rho=(f(|.|^2)-f(1))(\frac{1}{\omega}(|.|^2)\rho)\varphi. 
    \end{displaymath}
    Since
    \begin{displaymath}
        \langle \hat{u},(f(|.|^2)-f(1))\hat{\psi}\rangle_{\mathcal{S}'(\R^d),\mathcal{S}(\R^d)}=0, \quad \text{for all $\psi \in \mathcal{Z}(\R^d)$,}
    \end{displaymath}
    we have that
    \begin{displaymath}
        \langle \hat{u},(|.|^2-1)\varphi\rangle_{\mathcal{S}'(\R^d),\mathcal{S}(\R^d)}=\langle \hat{u},(f(|.|^2)-f(1))(\tfrac{1}{\omega}(|.|^2)\rho)\varphi\rangle_{\mathcal{S}'(\R^d),\mathcal{S}(\R^d)}=0,    
    \end{displaymath}
    from Remark \ref{rem:equi}. 
\end{proof}

We now consider the other direction. 

\begin{lemma}\label{Lemma:proof_1_2}
     Suppose that $u$ solves $(-\Delta)u=u$ in $\mathcal{S}'(\R^d)$. Then $u$ solves $f(-\Delta)u=f(1)u$ in the sense of Definition \ref{Def:Helmholtz}.
\end{lemma}
\begin{proof}
$0\notin \supp(\hat{u})$ by Lemma \ref{lemma:support}. Hence, from Remark \ref{rem:equi}, we wish to show that 
    \begin{displaymath}
       \langle \hat{u},(f(|.|^2)-f(1))\hat{\varphi}\rangle_{\mathcal{S}'(\R^d),\mathcal{S}(\R^d)}=0, \quad \text{for all $\varphi  \in \mathcal{Z}(\R^d)$}.     
    \end{displaymath}
  Observe that we have 
    \begin{align*}
        &\langle \hat{u},(f(|.|^2)-f(1))\hat{\varphi}\rangle_{\mathcal{S}'(\R^d),\mathcal{S}(\R^d)}\\
        &=\langle \hat{u},(f(|.|^2)-f(1))\rho\hat{\varphi}\rangle_{\mathcal{S}'(\R^d),\mathcal{S}(\R^d)}+\langle \hat{u},(f(|.|^2)-f(1))(1-\rho)\hat{\varphi}\rangle_{\mathcal{S}'(\R^d),\mathcal{S}(\R^d)},
    \end{align*}
        for $\varphi  \in \mathcal{Z}(\R^d)$.  
    Again, since $\supp(\hat{u})\subset \mathbb{S}^{d-1}$ by Lemma \ref{lemma:support}, we have that 
    \begin{displaymath}
        \langle \hat{u},(f(|.|^2)-f(1))(1-\rho)\hat{\varphi}\rangle_{\mathcal{S}'(\R^d),\mathcal{S}(\R^d)}=0, 
    \end{displaymath}
     Moreover, we have that 
    \begin{displaymath}
        (f(|.|^2)-f(1))\hat{\varphi} \rho=(|.|^2-1)(\omega(|.|^2)\rho\hat{\varphi}), \quad \text{for all $\varphi \in \mathcal{Z}(\R^d)$.} 
    \end{displaymath}
    Hence, we obtain
    \begin{align*}
        \langle \hat{u},(f(|.|^2)-f(1))\hat{\varphi}\rangle_{\mathcal{S}'(\R^d),\mathcal{S}(\R^d)}=\langle \hat{u},(|.|^2-1)(\omega(|.|^2)\rho\varphi)\rangle_{\mathcal{S}'(\R^d),\mathcal{S}(\R^d)}=0, 
    \end{align*}
    for all $\varphi \in \mathcal{Z}(\R^d).$
\end{proof}

\begin{proof}[First Proof of Theorem \ref{main_thm}]
    It is clear that 
    \begin{displaymath}
    V_{f(-\Delta)}=V_{(-\Delta)}, 
    \end{displaymath}
    from Lemma \ref{Lemma:proof_1_1} and \ref{Lemma:proof_1_2}. 
    Hence, from Theorem \ref{main_thm_2}, we have proved Theorem \ref{main_thm}. 
\end{proof}


\subsection{Second Proof} \label{SEC:second_proof}

From Lemma \ref{lemma:support} we have that $\supp{\hat{u}}\subset \mathbb{S}^{d-1}$. 

The goal of this proof will be to utilize the classical structure theorems. 

We now show that if $u$ solves the generalized helmholtz problem then $\hat{u}$ is necessarily a single-layer distribution on $\mathbb{S}^{d-1}$. 

We begin with the following lemma.

\begin{lemma}\label{lemma:proof_2_2}
    Suppose that $\hat{u}$ is a single-layer distribution on $\mathbb{S}^{d-1}$. 
    Then $u$ solves $f(-\Delta)u=f(1)u$ in the sense of Definition \ref{Def:Helmholtz}. 
\end{lemma}

\begin{proof}
Note if $\hat{u}$ is a single-layer distribution on $\mathbb{S}^{d-1}$ then $0\notin \supp(\hat{u})$. 
    By Remark \ref{rem:equi}, it is enough to show that  
    \begin{displaymath}
        \left \langle \hat{u},(f(|.|^2)-f(1))\hat{\varphi}\right \rangle_{\mathcal{S}'(\mathbb{R}^{d}),\mathcal{S}(\mathbb{R}^{d})}=0, \quad \text{for all $\varphi \in \mathcal{Z}(\mathbb{R}^d)$.}
    \end{displaymath}
Since, $\supp{\hat{u}}\subset \mathbb{S}^{d-1}$ it is enough to assume that $\hat{\varphi}\in \mathcal{D}(\R^d\setminus \{0\})$. 

If $\hat{u}$ is equal to $L^{(0)}_{\mathbb{S}^{d-1}}(T)$, where $T\in \mathcal{D}'(\mathbb{S}^{d-1})$, then we have that 
    \begin{align*}
         &\left \langle \hat{u},(f(|.|^2)-f(1))\hat{\varphi}\right \rangle_{\mathcal{S}'(\mathbb{R}^d),\mathcal{S}(\mathbb{R}^d)}\\
         &= \left \langle \hat{u},(f(|.|^2)-f(1))\hat{\varphi}\right \rangle_{\mathcal{D}'(\mathbb{R}^d),\mathcal{D}(\mathbb{R}^d)}\\
         &=\left \langle L^{(0)}_{\mathbb{S}^{d-1}}(T),(f(|.|^2)-f(1))\hat{\varphi}\right \rangle_{\mathcal{D}'(\mathbb{R}^d),\mathcal{D}(\mathbb{R}^d)}\\
         &=\left \langle T,(f(|.|^2)-f(1))\hat{\varphi}|_{\mathbb{S}^{d-1}}\right \rangle_{\mathcal{D}'(\mathbb{S}^{d-1}),\mathcal{D}(\mathbb{S}^{d-1})}=0,
    \end{align*}
    since $f$ is injective.
\end{proof}

To show that $\hat{u}$ is necessarily a single layer distribution on $\mathbb{S}^{d-1}$ we need the following lemmata which is inspired by \cite[Proof of Theorem 3.11]{MR1025883}. 

\begin{lemma}
    For a given $\tilde{\varphi}\in \mathcal{D}(\mathbb{R}^d)$, such that $\tilde{\varphi}|_{\mathbb{S}^{d-1}}=\varphi\in \mathcal{D}(\mathbb{S}^{d-1})$, then we have that the function $\gamma$, defined as,
    \begin{displaymath}
        \gamma(\xi):=\frac{1}{|\xi|^2-1}(\tilde{\varphi}(\xi)-\varphi(\xi/|\xi|)),
    \end{displaymath}
    for $\xi\in \mathbb{R}^d$ such that $|\xi|\in (1-\epsilon,1)\cup (1,1+\epsilon)$, has a smooth extension on $\{\xi\in \mathbb{R}^d:|\xi|\in(1-\epsilon,1+\epsilon)\}$. 
\end{lemma}

\begin{proof}
    Observe that by the Taylor theorem 
    \begin{align*}
        &\tilde{\varphi}(\xi)-\varphi(\xi/|\xi|)\\
        &=(\xi-\xi/|\xi|)\cdot \nabla\tilde{\varphi}(\xi/|\xi|)\\
        &+\int_{0}^{1}(1-t)(\xi-\xi/|\xi|)^T\nabla^2\tilde{\varphi}(\xi/|\xi|+t\xi)(\xi-\xi/|\xi|)\, \dt,\\
        &=(|\xi|-1)(\xi/|\xi|)\cdot \nabla\tilde{\varphi}(\xi/|\xi|)\\
        &+(|\xi|-1)\int_{0}^{1}(1-t)(\xi/|\xi|)^T\nabla^2\tilde{\varphi}(\xi/|\xi|+t\xi)(\xi-\xi/|\xi|)\, \dt,
    \end{align*}
    for $\xi\in \mathbb{R}^d$ such that $|\xi|\in (1-\epsilon,1)\cup (1,1+\epsilon)$
    where $\epsilon \in (0,1)$. Moreover, we have 
        \begin{align*}
        &\frac{1}{|\xi|^2-1}(\tilde{\varphi}(\xi)-\varphi(\xi/|\xi|))\\
        &=\frac{1}{1+|\xi|}(\xi/|\xi|)\cdot \nabla\tilde{\varphi}(\xi/|\xi|)\\
        &+\frac{1}{1+|\xi|}\int_{0}^{1}(1-t)(\xi/|\xi|)^T\nabla^2\tilde{\varphi}(\xi/|\xi|+t\xi)((\xi-\xi/|\xi|))\, \dt. 
    \end{align*}
    Hence, $\gamma$ has a smooth extension to $\{\xi\in \mathbb{R}^d:|\xi|\in(1-\epsilon,1+\epsilon)\}$. 
\end{proof}

\begin{lemma}\label{Lemma:proof_2}
 We have that  
    \begin{align*}
        \left \langle \hat{u},\tilde{\varphi}\right \rangle_{\mathcal{D}(\mathbb{R}^d),\mathcal{D}'(\mathbb{R}^d)}
        &=\left \langle \hat{u},E\varphi\right \rangle_{\mathcal{D}(\mathbb{R}^d),\mathcal{D}'(\mathbb{R}^d)},
    \end{align*}
    for $\tilde{\varphi}\in \mathcal{D}(\mathbb{R}^d)$ such that $\tilde{\varphi}|_{\mathbb{S}^{d-1}}=\varphi$. 
    
\end{lemma}

\begin{proof}
Note that
    \begin{align*}
        \rho\tilde{\varphi}&= \rho\varphi(./|.|)+ \rho(\tilde{\varphi}-\varphi(./|.|)),\\
        &=\rho\varphi(./|.|)+ (|.|^2-1)\rho\gamma,\\
        &=\rho\varphi(./|.|)+ (f(|.|^2)-f(1))\frac{1}{\omega}(|.|^2)\rho\gamma.
    \end{align*}
    Then it is clear that 
    \begin{align*}
        \left \langle \hat{u},\tilde{\varphi}\right \rangle_{\mathcal{D}(\mathbb{R}^d),\mathcal{D}'(\mathbb{R}^d)}&=\left \langle \hat{u},\rho\varphi(./|.|)\right \rangle_{\mathcal{D}(\mathbb{R}^d),\mathcal{D}'(\mathbb{R}^d)},\\
        &=\left \langle \hat{u},E\varphi\right \rangle_{\mathcal{D}(\mathbb{R}^d),\mathcal{D}'(\mathbb{R}^d)}. 
    \end{align*}
\end{proof}
\begin{proof}[Second Proof of Theorem \ref{main_thm}]
Define $\tilde{T}:\mathcal{D}(\mathbb{S}^{d-1})\rightarrow \C$ as 
\begin{displaymath}
    \langle \tilde{T},\varphi\rangle_{\mathcal{D}'(\mathbb{S}^{d-1}),\mathcal{D}(\mathbb{S}^{d-1})}:=\langle \hat{u},E\varphi\rangle_{\mathcal{D}'(\mathbb{R}^d),\mathcal{D}(\mathbb{R}^d)},
\end{displaymath}
for $\varphi \in \mathcal{D}(\mathbb{S}^{d-1})$, where $u$ solves $f(-\Delta)u=f(1)u$, in the sense of Definition \ref{Def:Helmholtz}. 
    Note that $\tilde{T}$ is well-defined on $\mathcal{D}(\mathbb{S}^{d-1})$, and is continuous on $\mathcal{D}(\mathcal{S}^{d-1})$ which implies that $\tilde{T}\in \mathcal{D}'(\mathbb{S}^{d-1})$. Moroever, by Lemma \ref{Lemma:proof_2},
    we have that for all $\tilde{\varphi}\in \mathcal{D}(\mathbb{R}^d)$ such that $\tilde{\varphi}|_{\mathbb{S}^{d-1}}=\varphi$ that 
    \begin{displaymath}
            \langle \hat{u},\tilde{\varphi}\rangle_{\mathcal{D}'(\mathbb{R}^d),\mathcal{D}(\mathbb{R}^d)}=\langle \hat{u},E\varphi\rangle_{\mathcal{D}'(\mathbb{R}^d),\mathbb{R}^d)}=\langle \tilde{T},\varphi\rangle_{\mathcal{D}'(\mathbb{S}^{d-1}),\mathcal{D}(\mathbb{S}^{d-1})}.
    \end{displaymath}
    
  Then $\hat{u}$ must necessarily be a single-layer distribution on $\mathbb{S}^{d-1}$.

    We know from  
    \cite[Proposition 8.4.1]{MR2961861} that 
    \begin{displaymath}
        u(x)=(2\pi)^{-d}\left \langle L_{\mathbb{S}^{d-1}}^{(0)}(\tilde{T}), e_{x}\right \rangle_{\mathcal{E}'(\R^d),\mathcal{E}(\R^d)}, \quad x\in \mathbb{R}^d.
    \end{displaymath} 
    Due to the support of $L_{\mathbb{S}^{d-1}}^{(0)}(\tilde{T})$ we have that 
    \begin{displaymath}
    \begin{split}
        u(x)&=(2\pi)^{-d}\left \langle L_{\mathbb{S}^{d-1}}^{(0)}(\tilde{T}), e_{x}\right \rangle_{\mathcal{E}'(\R^d),\mathcal{E}(\R^d)}\\
        &=(2\pi)^{-d}\left \langle L_{\mathbb{S}^{d-1}}^{(0)}(\tilde{T}), \rho e_{x}\right \rangle_{\mathcal{D}'(\R^d),\mathcal{D}(\R^d)}\\
        &=\left \langle T, e_{x}\right \rangle_{\mathcal{D}'(\mathbb{S}^{d-1}),\mathcal{D}(\mathbb{S}^{d-1})},
    \end{split}
    \end{displaymath}
    where $T:=(2\pi)^{-d}\tilde{T}.$ This combined with Lemma \ref{lemma:proof_2_2} we obtain our result. 
\end{proof}


\section{The eigenvalue problem on Banach spaces}
\label{SEC:BANACH}
Certain aspects of Section \ref{sec:first_proof} can be generalized. In particular, Lemma \ref{Lemma:proof_1_2}. Here we present a simple proof. 

To begin, let $(A,D(A))$ be a generator of a $C_0$-semigroup $(e^{-tA})_{t\geq 0}$ on a Banach space $X$.

For $u\in D(A)$ we define 
\begin{displaymath}
    f(A)u:=au+b(-Au)+\int_{(0,\infty)}(u-e^{-tA}u)\,d\mu(t).
\end{displaymath}
For a full characterization of $(f(A),D(f(A)))$ and its associated $C_0$-semigroup, see the celebrated Phillips Subordination Theorem \cite[Theorem 13.6]{MR2978140}. 

We are now ready to prove a generalization of Lemma \ref{Lemma:proof_1_2}. 

\begin{lemma}
    Suppose that $(-A)u=u$. Let $f$ be a Bernstein function with L\'evy triple $(a,b,\mu)$. Then we have that $f(A)u=f(1)u$.
\end{lemma}

\begin{proof}
    From \cite[Proposition 3.1.9 (j)]{MR2798103} we have that 
    \begin{displaymath}
        e^{-tA}u=e^{-t}u, \quad t\geq 0. 
    \end{displaymath}
    Since, $u\in D(A)$, we have that 
    \begin{displaymath}
    \begin{split}
        f(A)u&=au+b(-Au)+\int_{0}^\infty (u-e^{-tA}u)\,d\mu(t)\\
        &=au+bu+\int_{0}^\infty (u-e^{-t}u)\,d\mu(t)=f(1)u. 
    \end{split}
    \end{displaymath}
    
\end{proof}

The authors are not aware of a generalization of Lemma \ref{Lemma:proof_1_1}. Although, we suspect that such a result could be possible under some assumptions on $A$. 


\appendix

\section{The need for the space $\mathcal{Z}'(\R^d)$}
\label{SUBSEC:NEED_Z}
We now demonstrate why we are so particular about using $\mathcal{Z}'(\R^d)$ as opposed to tempered distributions $\mathcal{S}'(\R^d)$. The reason is that our operators $f(-\Delta)$ are not well-defined on $\mathcal{S}'(\R^d)$ in general. 

To explain this, we focus on the case on the fractional laplacian, $f(-\Delta)=(-\Delta)^\sigma$ where $\sigma \in (0,1).$ We also direct the reader to \cite[Section 2]{MR2270163} and \cite[Section 2]{MR4472712} for similar discussions. 

One potential definition of $(-\Delta)^\sigma$ on $\mathcal{S}'(\R^d)$ could be given by 
\begin{displaymath}
    \langle(-\Delta)^\sigma u, \varphi \rangle_{\mathcal{S}'(\R^d),\mathcal{S}(\R^d)}:=\langle u, \mathcal{F}^{-1}(|.|^{2\sigma}\hat{\varphi}) \rangle_{\mathcal{S}'(\R^d),\mathcal{S}(\R^d)}, 
\end{displaymath}
for all $(u,\varphi)\in \mathcal{S}'(\R^d)\times \mathcal{S}(\R^d)$.
The reader might believe that one can leverage the Parseval relation and duality but the problem here is that duality fails.  

In general $|.|^{2\sigma}\hat{\varphi} \notin \mathcal{S}(\R^d)$ if $\varphi \in \mathcal{S}(\R^d)$ and hence one cannot deduce that $\mathcal{F}^{-1}(|.|^{2\sigma}\hat{\varphi}) \in \mathcal{S}(\R^d)$. It now becomes problematic in giving meaning to the term 
\begin{displaymath}
\langle u, \mathcal{F}^{-1}(|.|^{2\sigma}\hat{\varphi}) \rangle_{\mathcal{S}'(\R^d),\mathcal{S}(\R^d)}, 
\end{displaymath}
since $u$ is a linear functional on $\mathcal{S}(\R^d)$. 

Another possibility is that we could define $\mathcal{F}((-\Delta)^\sigma u):=|.|^{2\sigma} \hat{u}$. However, the issue with this definition is that, \emph{\'a priori}, it is very possible that we are violating a rather important observation by Laurent Schwartz in that \emph{one cannot multiply distributions in general even if one of the distributions is a continuous function} \cite{MR64324}. 

To demonstrate this, we present a multiplication theorem which has recently gained exposure due to the celebrated \emph{theory of Regularity Structures} of Martin Hairer \cite[Proposition 4.14]{MR3274562}, see also \cite[Chapter 2, Section 8]{MR2768550}. Before we state the theorem, we denote $\mathcal{C}^{\alpha}_{loc}(\R^d)$ to be the space of (locally) $\alpha-$H\"older distributions \cite[Definition 3.7]{MR3274562}. 

\begin{theorem}\ \\
    Let $\alpha>0$ and $\beta<0$ such that $\alpha+\beta>0$. Then there exists a unique continuous map 
    \begin{displaymath}
        \mathcal{M}:\mathcal{C}^{\alpha}_{loc}(\R^d)\times \mathcal{C}^{\beta}_{loc}(\R^d)\rightarrow  \mathcal{C}^{\beta}_{loc}(\R^d),
    \end{displaymath}
    which extends the usual product $(f,g)\mapsto fg$ when $f$ is smooth.
\end{theorem}

The above theorem shows that one is quite limited in defining $|.|^{2\sigma}\hat{u}$, since $|.|^{2\sigma}\in \mathcal{C}^{2\sigma}_{loc}(\R^d)$ and hence we are restricting ourselves to the case where $\hat{u}\in\mathcal{C}^{-2\sigma+\epsilon}_{loc}(\R^d)$ where $\epsilon>0$. In the case where $u$ is a solution of the Helmholtz equation it is certainly possible that $\hat{u}\in \mathcal{C}^{-d}_{loc}(\R^d)$ (for instance the case where $\hat{u}=\delta_{\omega}$ where $\omega \in \mathbb{S}^{d-1}$). The above theorem certainly doesn't quite give us what we need to have a well-defined product. 

Nevertheless, it is definitely possible for us to give meaning to this product since, fortunately for us, the symbol of $(-\Delta)^{2\sigma}$ is smooth everywhere outside of $0$. This is why we use $\mathcal{Z}'(\R^d)$ instead of $\mathcal{S}'(\R^d)$ since one does not see the problematic behaviour of $|.|^{2\sigma}$ at the point $0$. The key observation is that whilst for $\varphi \in \mathcal{S}(\R^d)$ there is no guarantee that $\mathcal{F}^{-1}(|.|^{2\sigma}\hat{\varphi}) \in \mathcal{S}(\R^d)$, we do have that for $\varphi \in \mathcal{Z}(\R^d)$ then $\mathcal{F}^{-1}(|.|^{2\sigma}\hat{\varphi}) \in \mathcal{Z}(\R^d)$. Hence, the following definition of $(-\Delta)^\sigma:\mathcal{Z}(\R^d)'\rightarrow \mathcal{Z}'(\R^d)$:
\begin{displaymath}
    \langle(-\Delta)^\sigma u, \varphi \rangle_{\mathcal{Z}'(\R^d),\mathcal{Z}(\R^d)}:=\langle u, \mathcal{F}^{-1}(|.|^{2\sigma}\hat{\varphi}) \rangle_{\mathcal{Z}'(\R^d),\mathcal{Z}(\R^d)},
\end{displaymath}
for all $(u,\varphi)\in \mathcal{Z}'(\R^d)\times \mathcal{Z}(\R^d),$
is now well-defined. 

\begin{remark}
    There are certainly many variants of the multiplication theorem for distributions (too many to list here) but for the interested reader we direct them to \cite{caravenna2020hairers, dappiaggi2021microlocal, MR3684891, Broux2022} and references within. 
\end{remark}

\section{Pseudo-differential operators \\ induced from Bernstein functions}
\label{SEC:APPENDIX}

The aim of this section is to prove the following theorem. 
\begin{theorem}\label{thm:appendix}
Let $f$ be a Bernstein function with triple $(a,b,\mu)$.
The operator $f(-\Delta)$ defined by 
\begin{displaymath}
    \begin{split}
        f(-\Delta):&\mathcal{Z}(\R^d)\rightarrow \mathcal{Z}(\R^d),\\
        &\varphi\mapsto \mathcal{F}^{-1}(f(|.|^2)\hat{\varphi}),
    \end{split}
\end{displaymath}
is well-defined and continuous. Moreover, we have that the operator $f(-\Delta):\mathcal{Z}'(\R^d)\rightarrow \mathcal{Z}'(\R^d)$, defined by 
\begin{displaymath}
    \left \langle f(-\Delta)u,\varphi \right \rangle_{\mathcal{Z}'(\R^d),\mathcal{Z}(\R^d)}:=    \left \langle u,f(-\Delta)\varphi \right \rangle_{\mathcal{Z}'(\R^d),\mathcal{Z}(\R^d)},
\end{displaymath}
for all $(u,\varphi)\in \mathcal{Z}'(\R^d)\times \mathcal{Z}(\R^d),$ is well-defined and continuous. 
\end{theorem}

Throughout this section we assume, for simplicity, that $f$ is a Bernstein function with triple $(0,0,\mu)$ where $\mu$ is non-zero, since one can extend the result to the whole triple $(a,b,\mu)$ by linearity and continuity of $-\Delta$ on $\mathcal{Z}(\R^d)$ and $\mathcal{Z}'(\R^d).$ In fact, 
 more generally, we have that 
\begin{displaymath}
\begin{split}
    \partial^\alpha:&\mathcal{Z}(\R^d)\rightarrow \mathcal{Z}(\R^d),\\
    &\varphi \mapsto \partial^\alpha\varphi,
    \end{split}
\end{displaymath}
is a well-defined and continuous map for $\alpha \in \mathbb{N}_{0}^d. $

The proof is separated into 3 parts. 

The first part is the characterization of the derivatives of Bernstein functions which will aid us in some computations. 

The second part is focused on showing existence of smooth extensions of the function $f(|.|^2)\hat{\varphi}$ for $\varphi \in \mathcal{Z}.$

The last part focuses on the proof of Theorem \ref{thm:appendix}.  

Before we commence, an important observation on the space $\mathcal{Z}(\R^d)$ is that for $\varphi\in \mathcal{Z}(\R^d)$ we have, by the Taylor remainder theorem, that there exists a constant $C_N<\infty$ such that 
\begin{equation}\label{eqn:bounds_z}
    |\hat{\varphi}(x)|\leq C_N|x|^N, \quad \text{for all $|x|\leq 1$ and $N\in \mathbb{N}$}, 
\end{equation}
where
\begin{equation}\label{eqn:bounds_z_2}
    C_N:=\max_{|\alpha|=N}\sup_{y\in \mathbb{R}^d:|y|\leq 1}\left | \partial^\alpha\hat{\varphi}(y)\right |.
\end{equation}

For more details on the space $\mathcal{Z}(\R^d)$ and their relation to homogeneous spaces, see \cite[Chapter 1]{MR2768550} or \cite[Section 5.1.2]{MR3024598}.

\subsection{Characterization of derivatives of Bernstein functions}
We begin by characterizing the derivatives of Bernstein functions which will help us prove that $f(-\Delta):\mathcal{Z}(\R^d)\rightarrow \mathcal{Z}(\R^d)$ is well-defined. 

\begin{lemma}\label{Lemma:derivatives_bernstein}
    Then $n$-th derivative of $f$, denoted by $f^{(n)}$, is given by 
    \begin{equation}
        f^{(n)}(\lambda)=(-1)^{n-1}\int_{(0,\infty)}t^{n}e^{-\lambda t}\,\mu(dt), \quad \text{
        for $\lambda>0, n\in \mathbb{N}$}. 
    \end{equation} 
\end{lemma}

\begin{proof}
    First we check that the integral on the right hand side is well-defined. Observe that, for $\lambda>0$ and $n\in \mathbb{N}$, we have 
    \begin{displaymath}
        \int_{[0,\infty)}t^{n-1}e^{-\lambda t}\,\mu(dt)=\int_{(0,1]}t^{n}e^{-\lambda t}\,\mu(dt)+\int_{(1,\infty)}t^{n}e^{-\lambda t}\,\mu(dt).
    \end{displaymath}
    Note that 
    \begin{displaymath}
        \sup_{t>0}t^{n-1}e^{-\lambda t}<\infty, \quad \text{for all $\lambda>0$ and $n\in \mathbb{N}$}.
    \end{displaymath}
    Hence, we have that 
    \begin{displaymath}
        \int_{(0,\infty)}t^{n}e^{-\lambda t}\,\mu(dt)\leq \left (\sup_{t\in (0,1]}t^{n-1}e^{-\lambda t}\right )\int_{(0,1]}t \,\mu(dt)+\left (\sup_{t>1}t^{n}e^{-\lambda t}\right )\mu((1,\infty)).
    \end{displaymath}
    Using the fact that $\int_{(0,\infty)}(1\wedge t) \,\mu(dt)<\infty$ shows us that the integrals are well defined. Now let us identify the integrals. 
    
    The case $n=1$ is done in \cite[page 22]{MR2978140}. We proceed by an induction argument. Suppose the result is true for $n=k\in \mathbb{N}.$ Specifically, 
    \begin{displaymath}
        f^{(k)}(\lambda)=(-1)^{k-1}\int_{(0,\infty)}t^{k}e^{-\lambda t}\,\mu(dt), \quad \lambda>0.
    \end{displaymath}
    Observe that for $\lambda >\epsilon>0$, we have
    \begin{displaymath}
        \left | \frac{d}{d\lambda}\left (t^{k}e^{-\lambda t}\right ) \right |=\left | t^{k+1}e^{-\lambda t}\right |\leq t^{k+1}e^{-\epsilon t}, \quad t>0. 
    \end{displaymath}
    Hence, by differentiability of parameter integrals, we have that 
    \begin{displaymath}
        f^{(k+1)}(\lambda)=(-1)^{k}\int_{(0,\infty)}t^{k+1}e^{-\lambda t}\,\mu(dt), \quad \text{for all $\lambda>\epsilon$.}
    \end{displaymath}
    Since $\epsilon>0$ is arbitrary, we have our result by induction. 
\end{proof}

\begin{lemma}\label{Lemma:derivative_bounds}
    \begin{displaymath}
        \sup_{\lambda\in (0,1)}|\lambda^nf^{(n)}(\lambda)|<\infty, \quad \text{for all $n\in \mathbb{N}_0.$}
    \end{displaymath}
\end{lemma}

\begin{proof}
The case $n=0$ is trivial so we assume that $n\in \mathbb{N}.$

From Lemma \ref{Lemma:derivatives_bernstein} we have that 
\begin{align*}
    |\lambda^nf^{(n)}(\lambda)|&\leq \int_{(0,\infty)}(\lambda t)^ne^{-\lambda t}\,\mu(dt)\\
    &=\int_{(0,1]}(\lambda t)^ne^{-\lambda t}\,\mu(dt)+\int_{(1,\infty)}(\lambda t)^ne^{-\lambda t}\,\mu(dt).
\end{align*}
Note that
\begin{displaymath}
    \sup_{t>0}(\lambda t)^ne^{-\lambda t}=(\lambda t)^ne^{-\lambda t}|_{t=n/\lambda}=n^n e^{-n}. 
\end{displaymath}
Moreover, the function $t\mapsto (\lambda t)^ne^{-\lambda t}$ is increasing on $(0,n/\lambda]$. 
Hence,
\begin{displaymath}
    |\lambda^nf^{(n)}(\lambda)|\leq \int_{(0,1]}t\,\mu(dt)+n^ne^{-n}\mu((1,\infty)),
\end{displaymath}
for $\lambda\in (0,1).$
\end{proof}

\subsection{Existence of smooth extensions}

We now focus on establishing the existence of the smooth extension of $f(|.|^2)\hat{\varphi}$ for $\varphi \in \mathcal{Z}(\R^d).$ Since, it is obvious that $f(|.|^2)\hat{\varphi}$ is smooth away from $0$ we only focus on establishing the smooth extension to $0.$ 

The following lemmata will aid us in some computations. 

\begin{lemma}\label{Lemma:limit}
We have that 
\begin{displaymath}
    \lim_{x\rightarrow 0}\partial^\beta(f(|x|^2))\hat{\varphi}(x)=0,\quad \text{for all $\beta \in \mathbb{N}_0^d$, and $\varphi\in \mathcal{Z}(\R^d)$}. 
\end{displaymath}
\end{lemma}

\begin{proof}

It will be convenient for us to consider a more explicit form of $\partial^\beta(f(|x|^2))$. 

For $\beta_i\in \mathbb{N}_0$ we have, by the Fa\'a di Bruno formula, see \cite[Lemma 2.3]{MR2768550}, that 
    \begin{displaymath}
        \partial_{x_i}^{\beta_i}f(|x|^2)=p_{\beta_i}(x_i)f^{(\beta_i)}(|x|^2),
    \end{displaymath}
    for $(x_1,..,x_d)=x\in \mathbb{R}^d\setminus\{0\}$ and $\beta_i\in \mathbb{N}_0$, where $p_{\beta_i}$ is a polynomial of degree $\beta_i$. Hence, we have 
    \begin{equation}\label{eqn:leibniz}
        \partial^\beta f(|x|^2)=\partial_{x_1}^{\beta_1}...\partial_{x_d}^{\beta_d}f(|x|^2)=f^{(|\beta|)}(|x|^2)\prod_{i=1}^{d}p_{\beta_i}(x_i),
    \end{equation}
    for $(x_1,..,x_d)=x\in \mathbb{R}^d\setminus\{0\}$.
    From this, we have that 
    \begin{displaymath}
        |\partial^\beta(f(|x|^2))\hat{\varphi}(x)|=\left |\prod_{i=1}^{d}p_{\beta_i}(x_i)\right ||f^{(|\beta|)}(|x|^2)\hat{\varphi}(x)|. 
    \end{displaymath}
    However, from Lemma \ref{Lemma:derivative_bounds}, we have that 
    \begin{displaymath}
        |f^{(|\beta|)}(|x|^2)|
        \leq K|x|^{-2|\beta|}, \quad \text{for all $|x|\in (0,1)$}, 
    \end{displaymath}
    for some finite constant $K>0$. Moreover, from \eqref{eqn:bounds_z}, we have that 
    \begin{displaymath}
        |\hat{\varphi}(x)f^{(|\beta|)}(|x|^2)|\leq KC_{2\beta+1}|x|, \quad \text{for all $|x|\in (0,1)$},
    \end{displaymath}
    where $C_{2\beta+1}>0$ is also a finite constant which give us our result.

\end{proof}

\begin{lemma}\label{Lemma:smooth_extension}
    For $\varphi \in \mathcal{Z}(\R^d)$ we have that $f(|.|^2)\hat{\varphi}$ has a smooth extension at $0$. Specifically, we have that 
    \begin{displaymath}
        \lim_{x\rightarrow 0}\partial^\beta(f(|x|^2)\hat{\varphi}(x))=0, \quad \text{for all $\beta \in \mathbb{N}_0^d$.}
    \end{displaymath}
\end{lemma}
\begin{proof}
It is clear that $f(|.|^2)\hat{\varphi}$ is smooth on $\mathbb{R}^d\setminus\{0\}.$ However, to show the existence of a continuous extension of the partial derivatives at $0$ it will be useful for us to compute them explicitly. 

    By the multi-variable Leibniz rule, we have that 
\begin{displaymath}
    \partial^\beta(f(|x|^2)\hat{\varphi}(x))=\sum_{\alpha;\alpha \leq \beta} \binom{\beta}{\alpha} \partial^\alpha (f(|x|^2)) \partial^{\beta-\alpha} \hat{\varphi}(x),
\end{displaymath}
for $x\in \mathbb{R}^d\setminus \{0\}$ and $\beta\in \mathbb{N}_0^d$. However, since $\mathcal{F}^{-1}(\partial^{\beta-\alpha} \hat{\varphi})\in \mathcal{Z}(\R^d)$, we have our result by Lemma \ref{Lemma:limit}.
\end{proof}

\subsection{Proof of Theorem \ref{thm:appendix}}

\begin{proof}[Proof of Theorem \ref{thm:appendix}]
    To show that $f(-\Delta)\varphi \in \mathcal{Z}(\R^d)$ if $\varphi \in \mathcal{Z}(\R^d)$ it is sufficient for us to prove that $f(-\Delta)\varphi\in \mathcal{Z}(\R^d)$ since, if this is the case, we have from Lemma \ref{Lemma:smooth_extension} that $f(-\Delta)\varphi\in \mathcal{Z}(\R^d)$. By using the fact that $\mathcal{F}:\mathcal{S}(\R^d)\rightarrow \mathcal{S}(\R^d)$ is a continuous isomorphism it is sufficient for us to show that  
\begin{displaymath}
    \|f(|.|^2)\hat{\varphi}\|_{\alpha,\beta}:=\sup_{x\in \mathbb{R}^d} |x^\alpha\partial^\beta(f(|x|^2)\hat{\varphi}(x))|<\infty,
\end{displaymath}
for all $\varphi\in \mathcal{Z}(\R^d)$ and $\alpha,\beta\in \mathbb{N}_0^d$ since, if $\varphi \in \mathcal{Z}(\R^d)$ then $\mathcal{F}^{-1}(\partial^\alpha\hat{\varphi})\in \mathcal{Z}(\R^d)$ for $\alpha\in \mathbb{N}_0^d$. 

Note that since $\varphi \in \mathcal{Z}(\R^d)$ we have that 
\begin{displaymath}
    \sup_{x\in \mathbb{R}^d,|x|\leq 1} |x^\alpha\partial^\beta(f(|x|^2)\hat{\varphi}(x))|<\infty, \quad \text{for all $\alpha,\beta\in \mathbb{N}_0^d$},
\end{displaymath}
since $f(|.|^2)\varphi$ is smooth on $\mathbb{R}^d$, by Lemma \ref{Lemma:smooth_extension}.

We now want to show that, for $\varphi \in \mathcal{Z}(\R^d)$,
\begin{displaymath}
    \sup_{x\in \mathbb{R}^d,|x|\geq 1} |x^\alpha\partial^\beta(f(|x|^2)\hat{\varphi}(x))|<\infty, \quad \text{for all $\alpha,\beta\in \mathbb{N}_0^d$}.
\end{displaymath}

By the multi-variable Leibniz rule, as in the proof of Lemma \ref{Lemma:smooth_extension}, it is sufficient to prove that 

\begin{displaymath}
    \sup_{x\in \mathbb{R}^d,|x|\geq 1} |x^\alpha\partial^\beta(f(|x|^2))\hat{\psi}(x)|<\infty, \quad \text{for all $\alpha,\beta\in \mathbb{N}_0^d$ and $\psi\in \mathcal{Z}(\R^d)$}.
\end{displaymath}

Note that from \eqref{eqn:leibniz} we have that
\begin{equation}\label{eqn:ineq}
\begin{split}
    &\sup_{x\in \mathbb{R}^d:|x|>1}|x^\alpha\partial^\beta(f(|x|^2)\hat{\varphi}(x))| \\ 
    &\leq  \sup_{x\in \mathbb{R}^d:|x|>1} \left | x^\alpha \prod_{i=1}^{d}p_{\beta_i}(x_i) \hat{\varphi}(x)\right |\sup_{x\in \mathbb{R}^d:|x|>1}\left |f^{(|\beta|)}(|x|^2)\right | \\ 
    &\leq |f^{(|\beta|)}(1)| \sup_{x\in \mathbb{R}^d:|x|>1}\left | x^\alpha \prod_{i=1}^{d}p_{\beta_i}(x_i) \hat{\varphi}(x)\right |<\infty, 
    \end{split}
\end{equation}
since $\hat{\varphi}\in \mathcal{Z}(\R^d)$.
Hence, we have that $\varphi \in \mathcal{Z}(\R^d)$ then $f(-\Delta)\varphi \in \mathcal{Z}(\R^d)$.

We now focus on continuity of $f(-\Delta)$. 

Note, by duality, that continuity of $f(-\Delta):\mathcal{Z}(\R^d)\rightarrow \mathcal{Z}(\R^d)$ implies continuity of $f(-\Delta):\mathcal{Z}'(\R^d)\rightarrow \mathcal{Z}'(\R^d)$. Hence, we focus on proving the continuity of $f(-\Delta):\mathcal{Z}(\R^d)\rightarrow \mathcal{Z}(\R^d)$.

    Specifically, we show that if $\{\varphi_n\}_{n\in \mathbb{N}}\subset \mathcal{Z}(\R^d)$ such that $\lim_{n\rightarrow \infty}\varphi_{n}=0$ in $\mathcal{Z}(\R^d)$ then $\lim_{n\rightarrow \infty}f(-\Delta)\varphi_{n}=0$ in $\mathcal{Z}(\R^d)$. 
    
    By continuity of the fourier transform on $\mathcal{S}(\R^d)$, we have that this is equivalent to showing that 
    $\lim_{n\rightarrow \infty}f(|.|^2)\hat{\varphi}_{n}=0$ in $\mathcal{S}(\R^d)$.
    Specifically, we wish to show that 
    \begin{displaymath}
        \lim_{n\rightarrow \infty}\|f(|.|^2)\hat{\varphi}_n\|_{\alpha,\beta}=0, \quad \text{for all $\alpha,\beta \in \mathbb{N}_0^d$.}   
    \end{displaymath}

    By Leibniz rule, as in the proof of Lemma \ref{Lemma:smooth_extension}, it is sufficient to prove that 
\begin{displaymath}
    \lim_{n\rightarrow \infty}\sup_{x\in \mathbb{R}^d} |x^\alpha\partial^\beta(f(|x|^2))\hat{\psi}_n(x)|=0, \quad \text{for all $\alpha,\beta\in \mathbb{N}_0^d$},
\end{displaymath}
where $\lim_{n\rightarrow \infty}\psi_n=0$ in $\mathcal{Z}(\R^d)$. 

To begin, we consider the limit 
\begin{displaymath}
    \lim_{n\rightarrow \infty}\sup_{x\in \mathbb{R}^d, |x|\leq 1} |x^\alpha\partial^\beta(f(|x|^2))\hat{\psi}_n(x)|, \quad \text{for all $\alpha,\beta \in \mathbb{N}_0^d.$}
\end{displaymath}

Note that from \eqref{eqn:leibniz} and Lemma \ref{Lemma:derivatives_bernstein} we have that
\begin{align*}
    &\sup_{x\in \mathbb{R}^d:|x|\leq 1}|x^\alpha\partial^\beta(f(|x|^2)\hat{\psi}_n(x))|,\\
    &=  \sup_{x\in \mathbb{R}^d:|x|\leq 1} \left | x^\alpha f^{(|\beta|)}(|x|^2)\prod_{i=1}^{d}p_{\beta_i}(x_i)\right|\left |\hat{\psi}_n(x)\right |,\\
    &\leq \sup_{x\in \mathbb{R}^d:|x|\leq 1}  \left ||x|^{|\alpha|}f^{(|\beta|)}(|x|^2) \right| \left |\prod_{i=1}^{d}p_{\beta_i}(x_i)\right|\left |\hat{\psi}_n(x)\right |,\\
    &\leq \sup_{x\in \mathbb{R}^d:|x|\leq 1}  \left ||x|^{2|\beta|}f^{(|\beta|)}(|x|^2) \right| \left |\prod_{i=1}^{d}p_{\beta_i}(x_i)\right|\left |x^{|\alpha|-2|\beta|}\hat{\psi}_n(x)\right |,\\
    &\leq \sup_{x\in \mathbb{R}^d:|x|\leq 1}\left |x^{|\alpha|-2|\beta|}\hat{\psi}_n(x)\right | \sup_{x\in \mathbb{R}^d:|x|\leq 1}  \left ||x|^{2|\beta|}f^{(|\beta|)}(|x|^2) \right| \left |\prod_{i=1}^{d}p_{\beta_i}(x_i)\right|. 
\end{align*}

Observe that 
    \begin{align*}
        &\sup_{x\in \mathbb{R}^d:|x|\leq 1}  \left ||x|^{2|\beta|}f^{(|\beta|)}(|x|^2) \right| \left |\prod_{i=1}^{d}p_{\beta_i}(x_i)\right|\\
        &\leq         \sup_{x\in \mathbb{R}^d:|x|\leq 1}  \left ||x|^{2|\beta|}f^{(|\beta|)}(|x|^2) \right|  \sup_{x\in \mathbb{R}^d:|x|\leq 1} \left |\prod_{i=1}^{d}p_{\beta_i}(x_i)\right|<\infty,
    \end{align*}
    by Lemma \ref{Lemma:derivative_bounds}. If $|\alpha|\geq 2|\beta|$ then we have that 
    \begin{displaymath}
       \lim_{n\rightarrow \infty}\sup_{x\in \mathbb{R}^d:|x|\leq 1}\left  |x^{|\alpha|-2|\beta|}\hat{\psi}_n(x)\right |=0,
    \end{displaymath}
    since $\hat{\psi}_{n}\rightarrow 0$ in $\mathcal{S}.$ Otherwise, if $|\alpha|<2|\beta|$ then
    \begin{displaymath}
        \sup_{x\in \mathbb{R}^d:|x|\leq 1}\left  |x^{|\alpha|-2|\beta|}\hat{\psi}_n(x)\right |\leq \max_{|\delta|=2|\beta|-|\alpha|+1}\sup_{x\in \mathbb{R}^d:|x|\leq 1}\left | \partial^\delta\hat{\psi}_n(x)\right |,
    \end{displaymath}
    which implies that 
    \begin{displaymath}
       \lim_{n\rightarrow \infty}\sup_{x\in \mathbb{R}^d:|x|\leq 1}\left  |x^{|\alpha|-2|\beta|}\hat{\psi}_n(x)\right |=0,
    \end{displaymath}
    hence
    \begin{displaymath}
    \lim_{n\rightarrow \infty}\sup_{x\in \mathbb{R}^d, |x|\leq 1} |x^\alpha\partial^\beta(f(|x|^2))\hat{\psi}_n(x)|=0, \quad \text{for all $\alpha,\beta \in \mathbb{N}_0^d.$}
\end{displaymath}
The fact that 
    \begin{displaymath}
    \lim_{n\rightarrow \infty}\sup_{x\in \mathbb{R}^d, |x|> 1} |x^\alpha\partial^\beta(f(|x|^2))\hat{\psi}_n(x)|=0, \quad \text{for all $\alpha,\beta \in \mathbb{N}_0^d,$}
\end{displaymath}
follows from \eqref{eqn:ineq} and $\hat{\varphi}_{n}\rightarrow 0$ in $\mathcal{S}(\R^d)$. 
\end{proof}

\bibliographystyle{plain}

\end{document}